\documentclass[12pt]{amsart}

\usepackage{amsfonts, amstext, amsmath, amsthm, amscd, amssymb}
\usepackage{epsfig, graphics, psfrag, enumerate}
\usepackage{color}

\let\oldmarginpar\marginpar
\renewcommand\marginpar[1]{\oldmarginpar[\raggedleft\footnotesize #1]%
{\raggedright\footnotesize #1}}

\renewcommand{\setminus}{{\smallsetminus}}

\newcommand{\NN}{{\mathbb{N}}}

\newcommand{\cut}{{\backslash \backslash}}

\newcommand{\abs}[1]{{\left\vert #1 \right\vert}}

\newcommand{\GA}{{\mathbb{G}_A}}
\newcommand{\GB}{{\mathbb{G}_B}}
\newcommand{\G}{{\mathbb{G}}}

\newcommand{\HH}{{\mathbb H}}

\theoremstyle{definition}
\theoremstyle{plain}
\newtheorem{theorem}{Theorem}[section]
\newtheorem{corollary}[theorem]{Corollary}
\newtheorem{lemma}[theorem]{Lemma}

\newtheorem*{namedtheorem}{\theoremname}
\newcommand{\theoremname}{testing}

\theoremstyle{definition}
\newtheorem{define}[theorem]{Definition}
\newtheorem{remark}[theorem]{Remark}

\begin{document}
\title[On  the degree of the colored Jones polynomial]{On the degree of the  colored Jones polynomial}
\author[E. Kalfagianni]{Efstratia Kalfagianni}
\author[C. Lee]{Christine Ruey Shan Lee}

\address[]{Department of Mathematics, Michigan State University, East
Lansing, MI, 48824}

\email[]{kalfagia@math.msu.edu}

\email[]{leechr29@msu.edu}
\thanks{Lee was supported by NSF/RTG grant DMS-0739208  and NSF grant DMS--1105843.}
\thanks{{Kalfagianni was supported in part by NSF grants DMS--1105843 and DMS-1404754.}}

\thanks{ \today}

\begin{abstract} The extreme degrees of the colored Jones polynomial of \emph{any} link are bounded in terms of concrete data  from any link diagram. It is known that these bounds are sharp
for semi-adequate diagrams. One of the goals of this paper is to show the converse; if the bounds are sharp then the diagram is semi-adequate.
As a result, we use colored Jones link polynomials to extract an invariant that detects semi-adequate links and discuss some applications.\end{abstract}

\maketitle

\section{Introduction} The Jones polynomial and the colored Jones polynomials of \emph{semi-adequate} links
have been studied considerably in the literature \cite{lick-thistle,lickorish:book, stoimenow:coeffs} and  \cite{codyoliver1, armond, codyoliver,  dasbach-lin:head-tail,  garouvong, garouvong1,   garoufalidisLe} and they have been shown  to capture deep information about incompressible surfaces and geometric structures of link complements
 \cite{fkp:filling, fkp:PAMS, fkp:gutsjp, fkp:qsf, fkp:survey}.

The extreme degrees of the colored Jones polynomial of \emph{any} link are bounded in terms of concrete data  from any link diagram. It is known that these bounds are sharp
for semi-adequate diagrams. One of the goals of this paper is to show the converse; if the bounds are sharp then the diagram is semi-adequate.
As an  application we extract a link invariant, out of the colored Jones polynomial of a link, that detects precisely when the link is
semi-adequate. We discuss how this invariant can be thought of as generalizing certain  stable coefficients  of the colored Jones polynomials  of semi-adequate links,
studied by Armond \cite{armond},
Dasbach and Lin \cite{dasbach-lin:head-tail},  and Garoufalidis and Le \cite{garoufalidisLe}, to all links.
We also discuss how, combined with work of  Futer, Kalfagianni and Purcell \cite{ fkp:gutsjp, fkp:qsf}, our invariant detects certain incompressible
surfaces in link complements and their geometric types.

To describe the results and the contents of the paper in more detail, recall that a link is called  semi-adequate if it admits a link diagram that is $A$-\emph{adequate} or $B$-\emph{adequate}; see Definition \ref{semiad} for more details.
The colored Jones polynomial of a link $K$ 
is  a sequance of Laurent polynomial invariants $\{ J_K(n+1, q)\}_n$ such that $J_K(2, q)$ is the ordinary Jones polynomial. Let $d(n)$ and $d^{*}(n)$ denote the minimum and maximum degree of 
$J_K(n+1, q)$ in $q$, respectively.  It is known that for any  link diagram $D$ of $K$, there exists explicit 
functions $h_n(D)$ and $h^{*}_n(D)$ such that
$d(n)\geq h_n(D)$ and $d^{*}(n)\leq h^{*}_n(D)$; see \S3.1 for more details.
If $D$ is an $A$-adequate diagram then the lower bounds are sharp for all $n>1$ and similarly the upper bounds are sharp
if $D$ is $B$-adequate. 
It is known \cite{manchon}  that there exist  infinitely many links $K$  that admit diagrams $D$ with $d(1)=h_1(D)$ but $D$ or $K$ are not $A$-adequate. Thus the degree of the Jones polynomial alone doesn't detect semi-adequate links.
Our main result in this paper is the following theorem stating  that the degee of the colored Jones polynomial detects
semi-adequate links.

\begin{theorem} \label{zero1}Let $D$ be a diagram of a link $K$ and let  $h_n(D)$, $d(n)$ $h^{*}_n(D)$,  $d(^{*}n)$ and $J_K(n+1, q)$ be as above. Then,
  $D$ is $A$-adequate if and only if we have $d(n)=h_n(D)$, for some $n\geq 2$.

Similarly,   $D$ is $B$-adequate if and only if we have $d^{*}(n)=h^{*}_n(D)$, for some $n\geq 2$.
 \end{theorem}

Using Theorem \ref{zero1} and properties of semi-adequate diagrams we define 
a linear polynomial in $q$,  $T_K(q):= \alpha_K+ \beta_K q $, that is determined by $J_K(n+1, q)$,
and detects precisely 
when $K$ is $A$-adequate. More specifically we have the following:

\begin{corollary}\label{Adequate} There exists a link invariant $T_K(q)$, that is determined by the
colored Jones polynomial of $K$, such that
$T_K(q)\neq 0$ if and only $K$ is $A$-adequate. Furthermore,
if $T_K(q) = 1$, then $K$ is fibered.
\end{corollary}

Similarly, one can obtain a linear  polynomial in $q$,  $H_K(q)$, that is determined by $J_K(n+1, q)$,
and detects precisely 
when $K$ is $B$-adequate. 

To simplify the exposition, throughout the paper we will only deal with $A$-adequate links.
In Section two we state the definitions and recall the background and results from  \cite{dasbach-futer...}
that we need in this paper. We also prove some technical lemmas needed for the proofs of the main results.
In Section three we prove the above results and discuss some corollaries and applications.

\section{Ribbon graphs and Jones polynomials}

A \textit{ribbon graph} is a
multi-graph (i.e. loops and multiple edges are allowed) equipped with a
cyclic order on the edges at every vertex. Isomorphisms between ribbon graphs are  isomorphisms that preserve the given cyclic order of the edges. A \textit{ribbon graph}  can be  embedded on
an orientable surface such that every region in the complement of the graph is a disk \cite{bo-ri}.
We call the regions the \emph{faces} of the ribbon graph.
For a ribbon graph $\G$ we define the following quantities:  
\begin{align*}
e(\G) &= \text{the number of edges of $\G$} \\ 
v(\G) &= \text{the number of vertices of $\G$} \\ 
f(\G) &= \text{the number of faces of $\G$}\\
k(\G) &= \text{the number of connected components of $\G$} \\ 
g(\G) &= \frac{2k(\G) - v(\G) + e(\G) - f(\G)}{2}, \text{the \textit{genus} of $\G$} 
\end{align*}

A Kauffman state $\sigma$ on a link diagram $D$ is a choice of $A$--resolution or
$B$--resolution at each crossing of $D$.  For each state $\sigma$  of a  link diagram a ribbon graph  is constructed as follows: The result of applying $\sigma$
to $D$
is a collection of non-intersecting
circles in the plane, together with embedded arcs that record the crossing
splice. See Figure \ref{splicing}.
We orient  these circles in the plane by orienting each component clockwise or anti-clockwise according
to whether the circle is inside an odd or even number of circles, respectively.
The vertices,  of the ribbon graph correspond to the collection of circles and the edges
 to the crossings. The orientation of the circles defines the
orientation of the edges around the vertices. Throughout the paper we will adapt the convention that the vertices
of ribbobn graphs are state circles rather than points. Furthermore we will often use the term   ribbon graph  for
 the {\emph state graph}; the underlying graph with the orientations of the state circles ignored.
It is with this understanding that 
 the examples shown in the third panel of Figure \ref{splicing} as well as the graphs of Figures 3 and 4 
are called ribbon graphs.

 We will denote the ribbon graph associated to state $\sigma$ by $\G_\sigma$. 
For more details we refer the reader to   \cite{dasbach-futer...}.
Of particular interest for us will be the ribbon graphs $\GA$ and $\GB$ coming from the states with
all-$A$ splicings and all-$B$ splicings.

\begin{figure}[height=0.5in, width=0.5in]
 \includegraphics{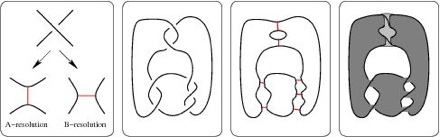}
\caption{From left to right: $A$-- and $B$--resolutions of a crossing, a link diagram, the ribbon graph $\GA$ and the surface $S_A$.
(Graphics by Jessica S. Purcell.)}
\label{splicing}
\end{figure}

\begin{define}
A \textit{spanning sub-graph} $\HH \subset \G$ of a ribbon graph  is obtained by removing edges from $\G$. 
\end{define}

With this setting we recall the following  \emph{spanning sub-graph expansion}  of the Kauffman bracket proven by
Dasbach, Futer, Kalfagianni, Lin and  Stoltzfus  \cite{dasbach-futer...}.

\begin{theorem} \cite{dasbach-futer...} \label{subgraphexpansion} Let $\GA$ be the ribbon graph obtained from the all $A$-resolution of a connected link diagram  $D(L)$. Then 
the Kauffman bracket of $D$ is given by
\begin{align*}
\langle D \rangle &= \sum_{\mathbb{H} \subset \GA} A^{e(\GA) - 2e(\mathbb{H})} (-A^2 - A^{-2})^{f(\mathbb{H}) - 1}, ƒ
\end{align*}
where $\mathbb{H} \subset \GA$ ranges over all spanning sub-graphs of $\GA$.
\end{theorem} 

\subsection{Semi-adequate links and ribbon graphs}

Lickorish and Thistlethwaite introduced the notion of $A$ and $B$-adequate links
and studied the properties of their link polynomials \cite{lick-thistle,lickorish:book}.  

\begin{define} \label{semiad}
A link diagram $D$ is called \emph{$A$--adequate} if  the ribbon graph $\GA$
corresponding to the all--$A$ state contains no 1--edge loops.  A link is called $A$-adequate if it admits
an $A$-adequate link diagram.

Similarly, $D$ is called $B$--adequate if the all--$B$ graph $\GB$ contains no 1--edge loops. 
 A link is called $B$-adequate if it admits
an $B$-adequate link diagram.

A link is called semi-adequate if it is $A$-adequate or $B$-adequate.
\end{define}

\medskip

\begin{define}\label{def:degree}For a connected link diagram $D$ and $n>1$ let $D^n$ denote  the diagram obtained from $D$ by taking $n$ parallel copies of $D$. 
Here the convention is that $D=D^1$.
See Figure \ref{example} for an example.
Define  

$$M(D):= e(\GA)+2v(\GA)-2, \ \  {\rm and} \ \  m(D):=-e(\GB)-2v(\GB)+2.$$
\end{define}

\begin{figure}[width=0.9in]
  \includegraphics{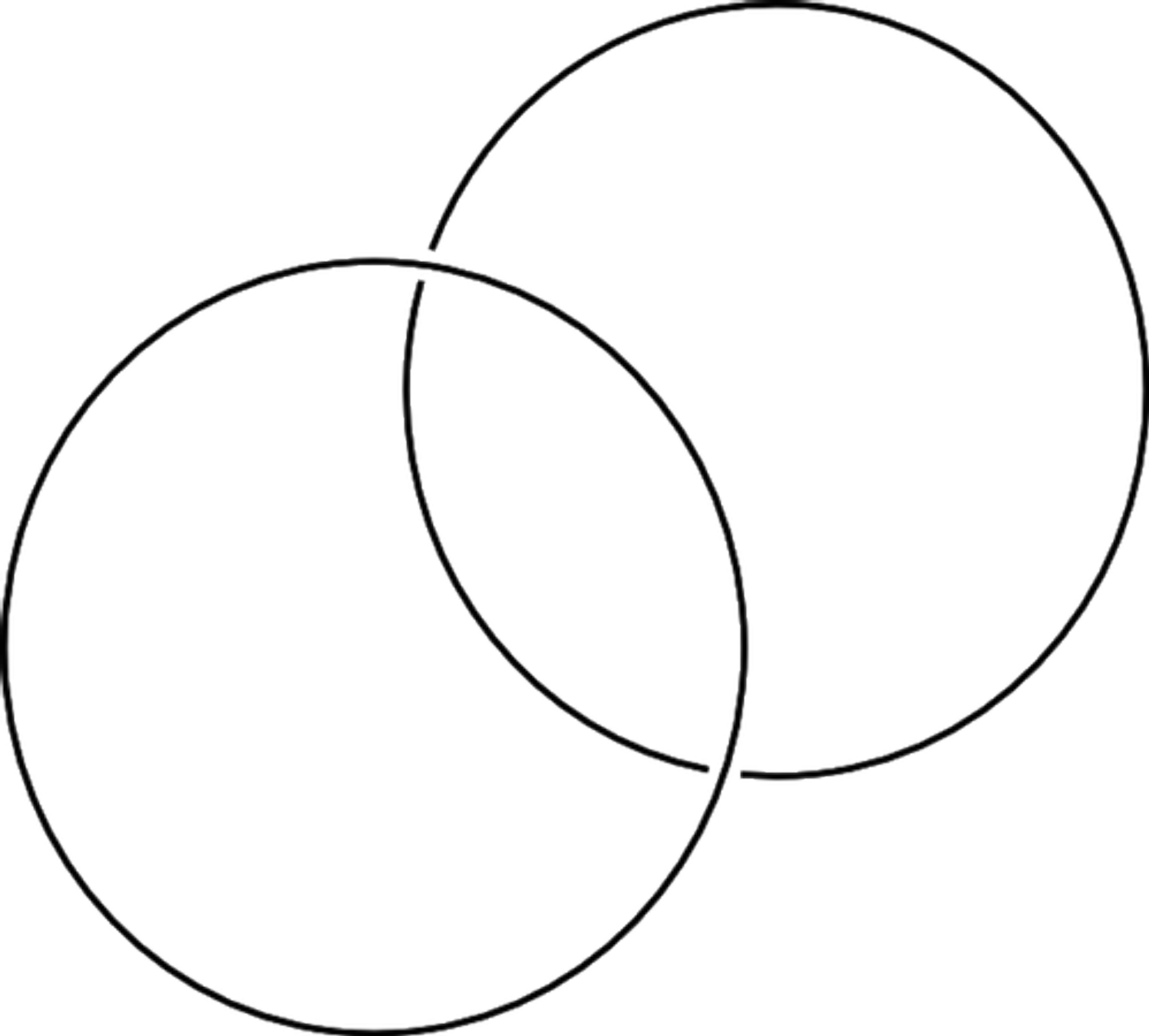}
  \hspace{.2in}
  \includegraphics{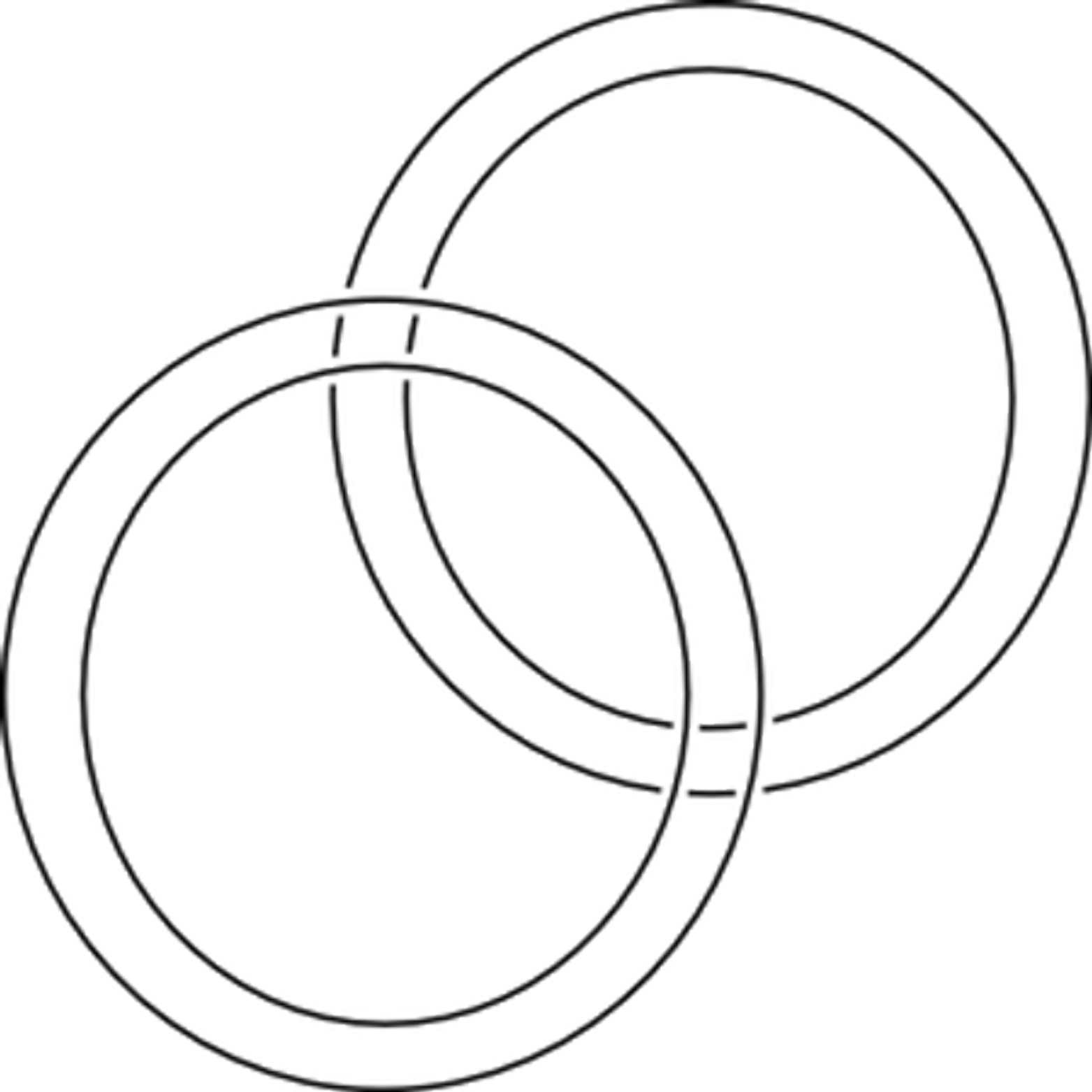}
  \caption{A link diagram and its  2-parallel}
  \label{example}
\end{figure}

With the notation of Definition \ref{def:degree}, it is known that, for every diagram $D$,  the highest degree (resp. the lowest degree ) of $\langle D \rangle$
is bounded above by $M(D)$ (resp. $m(D)$). Moreover, if
$D$ is $A$-adequate (resp. $B$-adequate) then this bound is sharp. See, for example,
\cite{lickorish:book}.

For $l\geq 0$, let $a_{M-l}{(D^n)}$ denote 
the  $l+1$-th coefficient $\langle D^n \rangle$, starting
from the maximum possible degree $M(D^n)$.
In this paper we are interested in the first two coefficients
$a_{M}{(D^n)}$ and  $a_{M-1}{(D^n)}$. For $A$-adequate links we have  $a_{M}{(D^n)}=\pm 1$,
for all $n>0$.
On the other hand, Manchon \cite{manchon}
shows that all integers can be realized as $a_{M}{(D)}$ for some link diagram.
Thus for $a_{M}{(D)}\neq 0,1, -1$,  Manchon's construction gives non $A$-adequate diagrams  for
which the upper bound on the degree of the Kauffman bracket  $\langle D \rangle$  is sharp.
In contrast to this, we show the following lemma which implies that the degree upper bound  of $\langle D^n \rangle$ is sharp, for some $n>1$,
if and only if $D$ is $A$-adequate.

\begin{lemma} \label{leading}
We have
that $a_{M}{(D^n)}\neq 0$, for some $n>1$, if and only if
 $D$ is  $A$-adequate. Equivalently, the diagram $D$ is not $A$-adequate
 if and only if $a_{M}{(D^n)}=0$, for all $n>1$.

\end{lemma}

\begin{proof}

It is known that if $D$ is $A$-adequate then $a_{M}{(D^n)}=\pm 1$; hence one direction of the lemma follows.
We will show that if $D$ is not $A$-adequate then $a_{M}{(D^n)}=0$, for all $n>1$.
For $n>1$ let $\GA^{n}$ denote the all-$A$ ribbon graph corresponding to $D^n$.
By Theorem \ref{subgraphexpansion}, the contribution of a spanning sub-graph $\mathbb{H} \subset \GA^n$ to $\langle D^n \rangle$ is given by 
\begin{equation}
X_{\HH}:=  A^{e(\GA^n) - 2e(\mathbb{H})} (-A^2 - A^{-2})^{f(\mathbb{H})-1}. \label{contribution} 
\end{equation}

A typical monomial of $X_{\HH}$ is of the form
$A^{e(\GA^n)-2e(\HH)+2f(\HH)-2-4s}$, for $0\leq s \leq f(\HH)-1$.
For a monomial to contribute to $a_{M}(D)$ we must have

$$e(\GA^n)-2e(\HH)+2f(\HH)-2-4s=e(\GA^n)+2v(\GA^n)-2,$$
or equivalently
$ f(\HH)=v(\GA^n)+e(\HH)+2s.$
Now we have
\begin{eqnarray*}
2g(\HH)&=&2k(\HH)-v(\GA^n) + e(\HH) - f(\HH)\\
&=& 2k(\HH)-2v(\GA^n) -2s,\\
\end{eqnarray*}
or  $g(\HH)=k(\HH)-v(\GA^n)-s\geq 0$.
But since $v(\GA^n)\geq k(\HH)$
(every component must have a vertex) and $s \geq 0$
we conclude that 
for a monomial of $X_{\HH}$ to contribute to $a_M(D)$ we must have
$s=g(\HH)=0$ and $v(\GA^n)= k(\HH)$.
Since  $f(\HH)=v(\GA^n)+e(\HH)$, the contribution of $\HH$ to $a_M(D)$ is
\begin{equation} (-1)^{f(\HH)-1}= (-1)^{v(\GA^n)+e(\HH)-1}. \label{topcontribution} \end{equation}

\begin{figure}[width=0.8in]
  \includegraphics{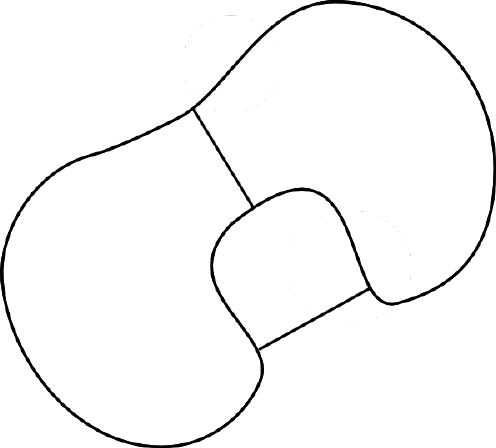}
  \hspace{.2in}
  \includegraphics{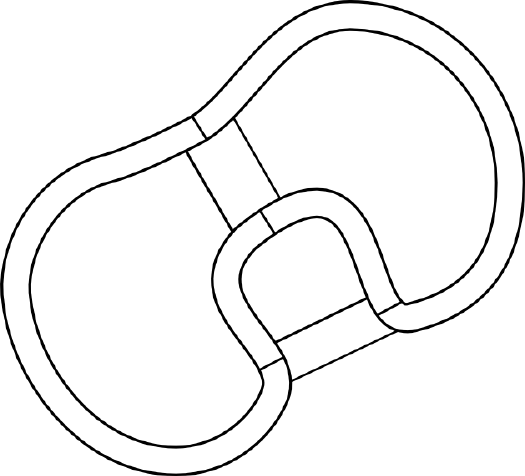}
  \caption{An example of ribbon graphs $\GA$ and $\GA^2$.  }
  \label{example1}
\end{figure}

Since $D$ is not $A$-adequate,
 $\GA$ must contain some loop edges. Thus $D^n$ also contains loop edges. Note that a sub-graphs of
 $\GA$ all of whose edges are loops may have positive genus. 
Nevertheless for $n>1$ every sub-graph of
 $\GA^n$ with only loop edges, must have genus zero as all the loop edges lie disjointly embedded
on the same side of the vertex they are attached to. See Figure  \ref{example1} for an example:
The graph $\GA$ in Figure 
\ref{example1} has genus one, since the two edges cannot be disjointly embedded on one side of the vertex. 
On the other hand in the graph $\GA^2$ all the loop edges are embedded on one side of some loop; thus every 
subgraph with only loop edges has genus zero.
Thus, if we let $\mathbb{L}_n  \subset \GA^n$ denote the  maximal spanning sub-graph whose edges are all the loops of  $\GA^n$
 then the  sub-graphs  of $\GA^n$ that contribute to $a_{M}{(D^n)}$ are in one to one correspondence with the sub-graphs of 
 $\mathbb{L}_n$.  Using equation  \ref{topcontribution}, It follows that, for $n>1$,
 
 \begin{equation}
a_{M}{(D^n)} = \sum_{\mathbb{H} \subset \mathbb{L}_n} (-1)^{e(\HH)+v(\GA^n)-1}=
(-1)^{v(\GA^n)-1}\left(\sum_{j=0}^{e(\mathbb{L}_n) }  {e(\mathbb{L}_n) \choose j} (-1)^j\right)=0.
\end{equation}
 
\end{proof}
Next we turn our attention to the second coefficient $a_{M-1}{(D^n)}$. We have the following. 

\begin{lemma} \label{second}  Suppose that  $a_{M}{(D)}\neq 0$ and that $D$ is not $A$-adequate. Then we have  $a_{M-1}{(D^n)}=0$, for all $n>2$.
\end{lemma}

\begin{proof}

For $n>2$ let $\GA^{n}$ denote the all-$A$ ribbon graph corresponding to $D^n$.

As in  \cite{DDD}, using  Theorem \ref{subgraphexpansion},  we see that   a spanning sub-graph $\mathbb{H} \subset \GA^n$ contributes to the term   $a_{M-1}{(D^n)}$
if and only if one of the following is true:
\begin{enumerate}

\item $v(\HH)=k(\HH)$ and $g(\HH)=1$.

\item  $v(\HH)=k(\HH)+1$ and $g(\HH)=0$.

\item $v(\HH)=k(\HH)$ and $g(\HH)=0$.

\end{enumerate}

If $v(\HH)=k(\HH)$ then all the edges of $\HH$ are loops.
Since $n>2$, as in the proof of Lemma \ref{leading}, 
subgraphs with only loop edges have genus 0.
Thus we 
cannot have any $\HH$ as in (1) above.
Before we turn our attention to types (2) or (3) we need the following:

\smallskip

{\it Claim.} Let   $\mathbb{H}\subset \GA^n$  be a spanning subgraph that
contains edges only on two vertices, say $v, v'$, and such that $v(\HH)=k(\HH)+1$. Then, there is a spanning sub-graph
 $\mathbb{L}_{\HH}\subset \GA^n$ all of whose edges are loops and none of the edges is attached to the vertices $v, v'$.
 See Figure \ref{example2} for an example of sub-graphs $\HH$ and $ \mathbb{L}_{\HH}$.
 
 \smallskip
 
 {\it  Proof of Claim.}  Since $D$ is not $A$-adequate,
 $\GA$ must contain some loop edges. Thus $D^n$ also contains loop edges.
 Since  $a_{M}{(D)}\neq 0$, $\GA$ contains a vertex, say $v_1$, that
  has loop edges attached to both sides; there are edges inside in $v_1$ and some loop edges
  outside of $v_1$, otherwise $a_M(D)$ would be equal to 0 (compare left picture of Figure \ref{example1}). In $\GA^n$,  $v_1$ will correspond to $n$ state circles (resp. vertices) and two of those state circles will have loops on them: one set of loops, originally coming from inside of the state circle corresponding to $v_1$ in $D$,  will be on the innermost state circle,
  while the loops coming from outside of $v_1$ will be on the outermost state circle  (compare right picture of Figure \ref{example1}). We will denote the two vertices corresponding to these state circles by $v_I$ and $v_O$, respectively. 
Since $n>2$,   in $ \GA^n$ we have at least one vertex $v\neq  v_I, v_O$ that also comes from cabling $v_1$. Moreover, there are no edges of $\GA^n$ with one end point on  $v_I$ and  the second on $v_O$.

  Let   $\mathbb{H} \subset \GA^n$  be a spanning subgraph that
contains  edges only between two vertices, say $v, v'$. By our discussion above, at least one of $v_I$ and $v_O$, say $v_I$,
is different from $v, v'$, and the claim follows.


\begin{figure}[width=0.8in]
  \includegraphics{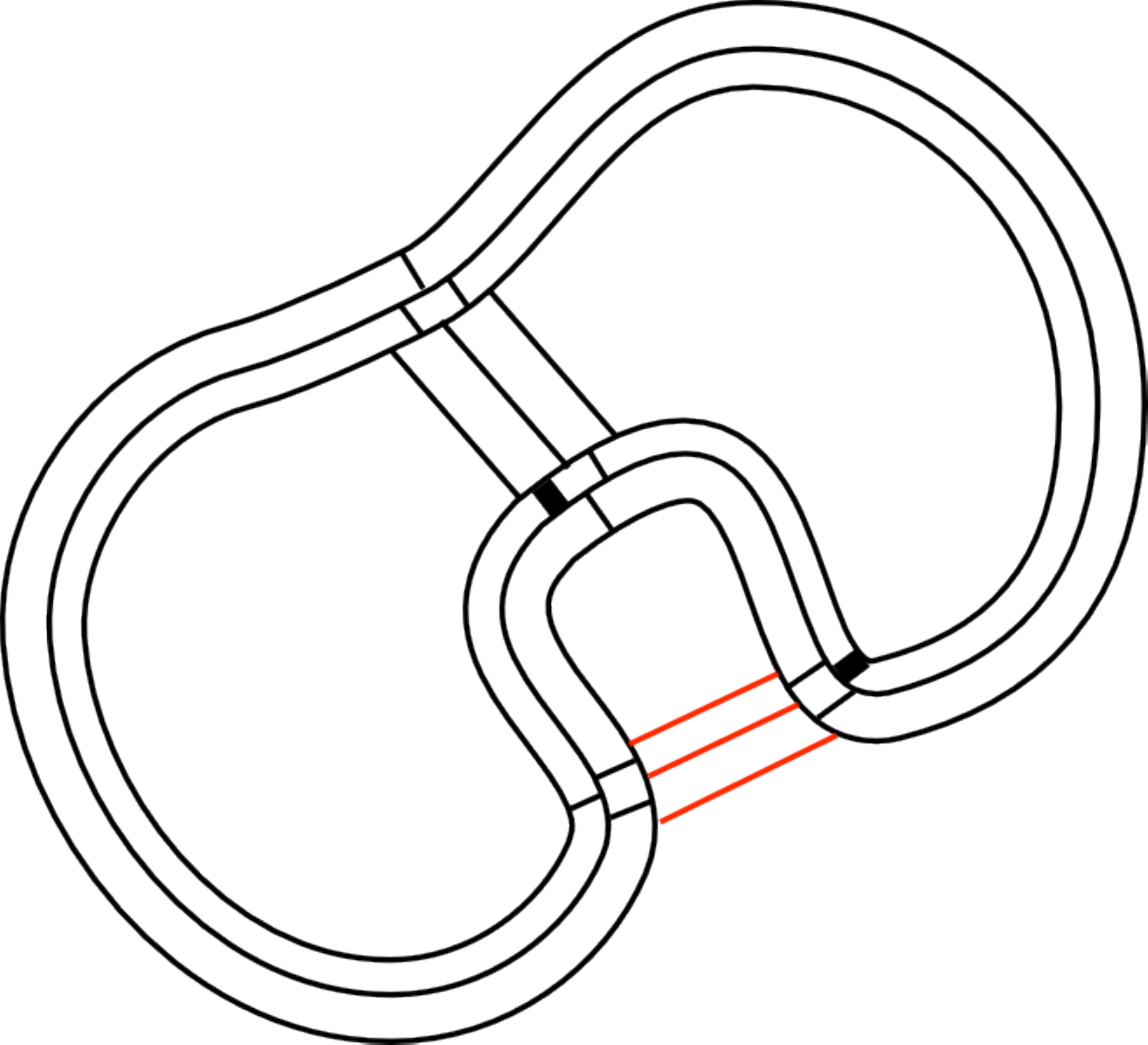}
  \caption{An example  a ribbon graphs  $\GA^3$: The edges of  a spanning sub-graph $H$  with $v(\HH)=k(\HH)+1$ are drawn thicker.
  The edges of the corresponding sub-graph  $\mathbb{L}_{\HH}$ are shown in red.}
  \label{example2}
  \end{figure}

Now we show that the contribution of all sub-graphs satisfying (2) above to $a_{M-1}(D^n)$ vanishes:  The condition $v(\HH)=k(\HH)+1$ 
means that there are exactly two vertices in $\HH$ with edges in $\HH$ joining them.
Consider the set $\mathcal H$ of sub-graphs of $\GA^n$ with edges on exactly two vertices, where we must have edges joining
these two vertices, and possibly loops on one of the vertices. 
Given  $\HH\in \mathcal H$,    let $\mathbb{L}_{\HH}\subset \GA^n$ denote the maximal (i. e. the one with the most edges)
sub-graph obtained from the claim above. That is,  all the edges of  $\mathbb{L}_{\HH}$ are loops that are attached 
to vertices disjoint from the two vertices containing the edges of $\HH$.
Hence adding any subset of edges of $\mathbb{L}_{\HH}$ to $\HH$ also produces a subgraph of type (2) above. 
Moreover, all the sub-graphs of type (2) are obtained from an element of $\mathcal H$ in this fashion.
The contribution of all type (2)  subgraphs to $a_{M-1}{(D^n)}$ is 

$$ \sum_{\HH \in \mathcal H} \left( \sum_{j=0}^{e(\mathbb{L}_{\HH})} {e(\mathbb{L}_{\HH}) \choose j}   (-1)^{f(\HH)+j}\right) =
\sum_{\HH \in \mathcal H}(-1)^{f(\HH)} \left( (1+(-1))^{e(\mathbb{L}_{\HH})}\right)=0.$$

  We now show that
  the contribution of type (3) sub-graphs to $a_{M-1}{(D^n)}$  is also zero.  As in the proof of Lemma \ref{leading},  let
  $\mathbb{L}_n  \subset \GA^n$ denote the  maximal spanning sub-graph whose edges are all the loops of  $\GA^n$. The subgraphs of type (3) are in one to one correspondence
  with the sub-graphs of $\mathbb{L}_n$. The contribution of these graphs
  is
  
  \begin{equation}(-1)^{v(\GA^n)-1}  \sum_{\mathbb{H} \subset \GA}  e(\HH) (-1)^{e(\HH)}, \label{cont} \end{equation}
  where $\HH$ ranges over all sub-graphs of $\GA^n$. We have
  
  \begin{eqnarray*}
 \sum_{\mathbb{H} \subset \GA}  e(\HH) (-1)^{e(\HH)}&=&   \sum_{j=0}^{e(\mathbb{L}_n)} \left( e(\mathbb{L}_n)-j\right)  {e(\mathbb{L}_n) \choose j}   (-1)^{e(\mathbb{L}_n)-j}       \\
&=&   e(\mathbb{L}_n)  \sum_{j=0}^{e(\mathbb{L}_n)} {e(\mathbb{L}_n) \choose j}   (-1)^{e(\mathbb{L}_n)-j}-     \sum_{j=1}^{e(\mathbb{L}_n)} j  {e(\mathbb{L}_n) \choose j}   (-1)^{e(\mathbb{L}_n)-j} \\
&=&  (-1)^ {e(\mathbb{L}_n)} \left( e(\mathbb{L}_n)   \sum_{j=0}^{e(\mathbb{L}_n)} {e(\mathbb{L}_n) \choose j}   (-1)^{j}+   \sum_{j=1}^{e(\mathbb{L}_n)} j  {e(\mathbb{L}_n) \choose j}   (-1)^{j-1}\right) \\
&=& (-1)^ {e(\mathbb{L}_n)} \left(e(\mathbb{L}_n)  f(-1) + f'(-1)\right)=0,\\
\end{eqnarray*}
where  $f(x):=(1+x)^{{e(\mathbb{L}_n) }}$ and $f'(x)= e(\mathbb{L}_n) (1+x)^{{e(\mathbb{L}_n) -1}}$, the first derivative of $f(x)$.
Thus the  quantity in equation \ref{cont} is zero as desired.
  
  \end{proof}

\begin{remark} {\rm The assumption $a_{M}{(D)}\neq 0$ in the statement of Lemma \ref{second}
is not necessary. However the proof given above does not go through without the restriction that 
$a_{M}{(D)}\neq 0$ . In the general case the proof will follow from the arguments in \cite{lee}}.
\end{remark}

  \section{ Colored Jones polynomial relations}
  A good reference for the following discussion is Lickorish's book \cite{lickorish:book}.
   The colored Jones polynomials of a link $K$ have a convenient expression in terms of \emph{Chebyshev polynomials}. For $n \geq 0$, the polynomial $S_n(x)$ is defined recursively as follows:
\begin{equation}\label{eq:cheb-recursive}
S_{n+1}(x) = x S_n(x) - S_{n-1}(x), \qquad S_1(x) = x, \qquad S_0(x) = 1.
\end{equation}

Let $D$ be a diagram of a link $K$.  Recall that for  an integer $m > 0$,  $D^m$ denotes the diagram obtained from $D$ by taking $m$ parallel copies of $K$.
This is the $m$--cable of $D$ using the blackboard framing; if $m=1$ then $D^1=D$. 
Recall that $\langle D^m \rangle$ denotes the Kauffman bracket of $D^m$.
Let  $c_+=c_+(D)$ and $c_-=c_-(D)$  denote the number of positive and negative crossings of $D$ and let
$w=w(D) = c_+ - c_-$ denote the writhe of $D$. Then we may define the function
\begin{equation}\label{eq:unreduced}
G_D(n+1, A):= \left( (-1)^n A^{n^2+2n} \right)^{-w} (-1)^{n-1} \langle S_n( D) \rangle,
\end{equation}
where $S_n(D)$ is a linear combination of blackboard cabling of $D$, obtained via equation (\ref{eq:cheb-recursive})
and the bracket is extended linearly.  For the results below, the important corollary of the recursive formula for $S_n(x)$ is that
\begin{equation}\label{eq:chebyshev}
S_n(D) = D^n + (1-n)D^{n-2} + \mbox{lower degree cablings of } D.
\end{equation}

The reduced  $(n+1)$--colored Jones polynomial  of $K$, denoted by
$J_K(n+1, q)$, is obtained from $ \left(\frac{A^4 - A^{-4}}{A^{2n} - A^{-2n}} \right) G(n+1, A)$ by substituting $q:=A^{-4}$.
 The ordinary Jones polynomial corresponds to $n=1$.

 \subsection{ Bounds on the degree of colored Jones polynomials}
 Given a  link diagram $D(L)$, let $v_A(D)$ denote the number of components resulting from $D$ by applying the all-$A$ Kauffman state.
For $n\in \NN$ let
$$h_n(D):= 2{c_-(D)}n^2+ 2\left(v_A(D)-w(D)\right) n -2. $$ 

Let $d(n)$ denote the maximum degree in $A$ of $G_D(n+1, A)$ .
It is know that, for every $n>0$,  $d(n)\leq h_n(D)$ and that if $D$ is a $A$-adequate diagram then

\begin{equation}
d(n)=h_n(D)= 2{c_-(D)}n^2+ 2\left(v_A(D)-w(D)\right)n -2,\label{degree}
\end{equation}
and the coefficient of this leading terms is known to be $\pm 1$.
For $n=1$ the equation \ref{degree}  is not enough to characterize $A$-adequate diagrams:  Manchon \cite{manchon}
shows that all non-zero integers can be realized as leading coefficients of Jones polynomials of knots with diagrams satisfying equation \ref{degree}. 
Links realizing integers $a_{M}\neq -1, 1$ will necessarily be
non $A$-adequate. It follows that there exist  infinitely many links $L$  that admit diagrams $D$ with $d(1)=h_1(D)$ but $D$ or $L$ is non-adequate.
In contrast to this, in this paper we have:

\begin{theorem} \label{zero}Let $D$ be a diagram of a link $K$ and let  $h_n(D)$, $d(n)$ and $G_D(n+1, A)$ be as above. Then,
  $D$ is $A$-adequate if and only if we have $d(n)=h_n(D)$, for some $n\geq 2$.

Furthermore, we have $ d(n)=h_n(D)$, for every $n\geq 2$, if and only if $d(2)=h_2(D)$.
\end{theorem}
\begin{proof} By equations \ref{eq:unreduced} and \ref{eq:chebyshev}
the coefficient of $A^{h_n(D)}$ in $G_D(n+1, A)$ is the same, in absolute value,  as $a_{M(D^n)}$.
Thus, the theorem follows immediately by Lemma \ref{leading}.
\end{proof}

\begin{remark} {\rm Theorem \ref{zero} should be compared with the main results of \cite{thi:adequate} where Thistlethwaite
shows that part of the 2-variable Kauffman polynomial where the  total degree is the number of crossings in a link diagram,
is non-zero if and  only if the diagram is $A$-adequate.}
\end{remark}

\subsection{Stable invariants}
On the set of oriented link diagrams consider the complexity $$\left( c_-(D), \ c(D), \  v_A(D)- w(D)\right),$$
ordered lexicographically.
For a link $K$ define ${\mathcal D}(K)$ to be the set of diagrams representing $K$ and minimize this complexity.  
More specifically, we define $ {\mathcal D}(K)$  as follows: First consider the set of all diagrams of $K$ that minimize the number 
of negative crossings $c_{-}$; call this minimum number $c_{-}(K)$. Then, within this set restrict to the subset say $X_K$ of diagrams that have minimum crossing number: that is $D\in X_K$ if and only if $c(D)\leq c(D')$, for all diagrams of $K$ with $c_{-}(D)=c_{-}(K)$.
Since there are only finitely many diagrams of bounded crossing number, given a link $K$,  the set $ X_K$ is finite. 
Thus we may define

$$ {\mathcal D}(K):=\{ D\in X_K\ | \  v_A(D)-w(D)\leq v_A(D')-w(D'), {\rm for} \ {all} \ D' \in X_K \}.$$
\begin{lemma} Suppose that for a link $K$, there is $D\in {\mathcal D}(K)$ that is $A$-adequate. Then, all the diagrams in ${\mathcal D}(K)$ are $A$-adequate.

\end{lemma}\label{setD}
\begin{proof}Suppose that $D\in {\mathcal D}(K)$  is $A$-adequate and let $D'$ be another diagram in $D\in {\mathcal D}(K)$.
Since $D,D'\in {\mathcal D}(K)$, we have
$ c_-(D)= c_-(D')$ and  $v_A(D)-w(D)=v_A(D')-w(D')$. Thus $ h_n(D)= h_n(D')$.
Since $D$ is $A$-adequate, we have $d(n)= h_n(D)= h_n(D')$, for all $n>0$.
Thus, by Theorem \ref{zero}, $D'$ is $A$-adequate. 
\end{proof}

\begin{define}\label{defi:stable}Given a link $K$ and a link  diagram  $D\in  {\mathcal D}(K)$ we define  $T_{(D, n)}(q)$ as follows:
For $n>2$, define
 \begin{equation} \alpha(D, n):=\abs{ a_{M}{(D)} a_{M}{(D^n)} }\ \ \ \ {\rm and} \ \ \ \
 \beta(D, n):=\abs{a_{M}{(D)}a_{M-1}(D^n)}. \label{definvariant} \end{equation}
Now let $T_{(D, n)}(q):=  \alpha(D, n) + \beta(D, n) q$.
\end{define}
 Now we will show that the quantities defined in equation  \ref{definvariant}  are in fact independent of $n$ and of the diagram  $D$. 
 
\begin{corollary} \label{invariant} The quantities, $\alpha(D, n)$  and $\beta(D, n)$ defined above are  independent of the diagram  $D\in  {\mathcal D}(K)$ and
of $n$. Thus they are invariants of $K$ denoted by $\alpha_K$ and $\beta_K$.

\end{corollary}
\begin{proof} Suppose that $K$ is not $A$-adequate and let $D\in  {\mathcal D}(K)$. Then  by Theorem \ref{zero} we have
$a_{M}{(D^n)}=0$ for every $n>1$ and every  $D\in  {\mathcal D}(K)$.
Thus by equation \ref{definvariant}
we have $\alpha(D,n)=0$. If $a_M(D)=0$, then equation \ref{definvariant} implies $\beta(D,n)=0$.
Suppose $a_M(D)\neq 0$. Then
by Lemma \ref{setD}, $a_{M-1}{(D^n)}=0$, thus $\beta(D, n)=0$.
Thus, in this case, the definition  $\alpha(D, n)=\beta(D, n)=0$, for every $n>1$
and $D\in  {\mathcal D}(K)$.

Suppose now that $K$ is $A$-adequate. By Lemma \ref{setD}, for every $D\in  {\mathcal D}(K)$, $D$ is an $A$-adequate diagram.
Then, by  \cite{lickorish:book} we have   $ \alpha(D, n)=1$, for all $n>0$. Similarly 
$\beta(D, n)$ is the absolute value of the penultimate coefficient of $J_K(n+1, q)$ and thus an invariant of $K$.
By \cite{dasbach-lin:head-tail}, $\beta(D, n)$ is also independent of $n$.
\end{proof}

Define the linear polynomial in $q$,  $T_K(q):= \alpha_K+ \beta_K q $. This is an invariant of $K$ that 
 detects exactly when $K$ is $A$-adequate. More specifically we have the following:

\begin{corollary}\label{Adequate} We have $T_K(q)\neq 0$ if and only $K$ is $A$-adequate. Furthermore,
if $T_K(q) = 1$, then $K$ is fibered.
\end{corollary}
\begin{proof}  As already mentioned in the proof of Corollary \ref{invariant},  if $K$ is not $A$-adequate, then $T_K(q)=0$.
On the other hand, if $K$ is $A$-adequate then we know that $\alpha_K=1$ and thus $T_K(q)\neq 0$.

Suppose now that $T_K(q)=1$. Then, in particular, $\beta_K=0$. By Corollary 9.16, of \cite{fkp:gutsjp}, $K$ has to be fibered.
\end{proof}

\subsection{ Stabilization properties of Jones polynomials}  The coefficients of the colored Jones polynomials of $A$-adequate links have stabilization properties that have been studied by several authors in the recent years \cite{dasbach-lin:head-tail, armond,  garoufalidisLe}.  Dasbach and Lin observed that the last three coefficients of $J_K(n+1, q)$ stabilize. Armond \cite{armond}
and Garoufalidis and Le \cite{ garoufalidisLe} generalized this phenomenon to show the following: For every $i>0$ the $i$-th to last
coefficient of  $J_K(n+1, q)$, stabilizes  for $i\geq n$. 
These coefficients can be put together to form the \emph{tail} of the colored Jones polynomial.  In the case of $A$-adequate links the  invariants $\beta_K, \alpha_K$ 
defined above, are the last couple of  stable coefficients.
 In fact, Garoufalidis and Le  have studied the ``higher order" stability properties of the colored Jones polynomials and they showed that the stable coefficients 
 of the  polynomials $J_K(n+1, q)$ give rise to infinitely many $q$-series with interesting number theoretic and physics connections.  
 On the other hand, the work of Futer, Kalfagianni  and Purcell \cite{fkp:filling, fkp:PAMS, fkp:gutsjp, fkp:qsf, fkp:survey}
showed that certain stable coefficients of  $J_K(n+1, q)$ encode information about incompressible surfaces in knot complements and their
geometric types and have direct connections to hyperbolic geometry. See also discussion below.
Rozansky \cite{rozansky} showed that the stability behavior also appears in the categorifications of the colored Jones polynomials \cite{rozansky}.

The structure of the colored Jones polynomials of non-semi-adequate links and its geometric content are much less understood. 
In a forthcoming paper \cite{lee},  C. Lee generalizes Theorem \ref{zero} to show the following:  If $D$ is not $A$-adequate, then, we have  $d(n)\leq h_n(D)-(n-1)$, for every $n\geq 2$. This implies that the first $n-1$ coefficients 
of  $G_D(n+1, A)$, starting from the one for degree $h_n(D)$, are zero, for every $n\geq 2$.
Given a link $K$ and a link  diagram  $D\in  {\mathcal D}(K)$ one can define a power series $J^0=J^0(D)$ as follows:
Define $\beta_1=\beta_1(D)$ to be the coefficient of $A^{h_2(D)}$ in $G_D(3, A)$. For $i>1$, define
$\beta_i=\beta_i(D)$ to be the  coefficient of $A^{h_{i+1}(D)-4(i-1)}$.
Now let
$$J^0_K(q)=\sum_{i=1}^{\infty} \beta_i q^{i-1}.$$
Lee shows that $J^0_K(q)\neq 0$,
 if and only if $K$ is $A$-adequate.
This shows that the coefficients of  $J_K(n+1, q)$, at the level where the tail of semi-adequate links occurs,  also stabilize
but the corresponding tail is trivial. This was conjectured by Rozansky  
in  \cite{rozansky} where he also conjectures that this behavior should persist in the setting of categorification (Conjecture 2.13 of \cite{rozansky}).

\subsection{Detecting incompressible surfaces and their geometric types} For every
$D\in  {\mathcal D}(K)$ 
we obtain a surface $S_A$, as follows. Each state circle of $v_A(D)$ bounds a disk in $S^3$. This collection of disks can be disjointly embedded in the ball below the projection plane. At each crossing of $D$, we connect the pair of neighboring disks by a half-twisted band to construct a surface $S_A \subset S^3$ whose boundary is $K$. 
See Figure \ref{splicing} for an example. By the work of the first author with Futer and Purcell  \cite{fkp:gutsjp, fkp:qsf}, the invariant  $T_K(q)$ detects the geometric types
of the surface $S_A(D)$ and contains a lot of information about the geometric structures of 
of the complements $S^3\cut S_A(D)$ and $S^3\setminus K$. For example, combining Corollary \ref{Adequate} with results of 
 \cite{fkp:gutsjp, fkp:qsf} we have the following; for terminology and more details the reader is referred to the original references.

\begin{corollary} \label{properties} The invariant $T_K(q)$ has the following properties:
\begin{enumerate}
\item For every   $D\in  {\mathcal D}(K)$, the surface $ S_A(D)$ is essential (i.e. $\pi_1$-injective) in $S^3\setminus K$ if and only if
$T_K(q)\neq 0$.
\item For every   $D\in  {\mathcal D}(K)$, the surface $ S_A(D)$ is a fiber in the complement $S^\setminus K$ 
 if and only if
$T_K(q)=1$.

\item Suppose that $K$ is hyperbolic. Then, for every   $D\in  {\mathcal D}(K)$, the surface $ S_A(D)$ is quasifuschian $S^3\setminus K$ if and only if
$T_K(q)\neq 0, 1$.

\end{enumerate}
\end{corollary}

\begin{proof} By Theorem  3.19 of \cite{fkp:gutsjp}  $S_A(D)$  is essential precisely when $D$ is $A$-adequate.
Thus part (1) follows from Corollary \ref{Adequate}.
For part (2), first note that  if $T_K(q)=1$ then, by Theorem \ref{zero}, $K$ is $A$-adequate.
Thus, $\beta_K$ is, in absolute value,  the penultimate  stable  coefficient of the colored Jones polynomial in the sense
of  \cite{dasbach-lin:head-tail}.
Now  by \cite{fkp:gutsjp}, $\beta_K=0$ if and only if  $ S_A(D)$ is a fiber in the complement $S^\setminus K$ .
Finally for (3) we note  that $T_K(q)\neq 0$, then again by Theorem \ref{zero}, if and only if  $K$ is $A$-adequate.
Now  $T_K(q)\neq 1$,
if and only $\beta_K\neq 0$.
Then, by Theorem 1.4 of \cite{fkp:qsf}  $\beta_K\neq 0$ if and only if the surface $ S_A(D)$  is quasifuchsian 
for every $A$-adequate diagram $D$.
\end{proof}

\smallskip

{\bf Acknowledgement.}  This work started while the authors were attending the conference 
Quantum Topology and Hyperbolic Geometry in Nha Trang, Vietnam (May 13-17, 2013).
We thank the organizers, Anna Beliakova,  Stavros Garoufalidis, Phung Hai, 
Vu  Khoi,
Thang Le, Chu  Loc,  and Phan Phien
for their hospitality and for providing excellent working conditions. We also thank Lev Rozansky for a useful conversation during the same conference.

\bibliographystyle{hamsplain} \bibliography{biblio}
\end{document}